\documentclass{amsart}

\usepackage[T1]{fontenc}
\usepackage[utf8]{inputenc}
\usepackage[english]{}
\usepackage{tikz}
\usepackage{amsmath}
\usepackage{amssymb,amsthm}
\usepackage{mathrsfs}
\usepackage{stmaryrd}
\usepackage{enumerate}
\usepackage{fancyhdr}
\usepackage{array}
\usepackage{hyperref}

\newtheorem{theo}{Theorem}
\newtheorem{defi}{Definition}
\newtheorem{coro}{Corollary}
\newtheorem{lemm}{Lemma}

\newtheorem{prop}{Proposition}

\newcommand{\Z}{{\mathbb Z}}
\newcommand{\Q}{{\mathbb Q}}

\newcommand{\R}{{\mathbb R}}
\newcommand{\F}{{\mathbb F}}

\renewcommand{\P}{{\mathbb P}}
\newcommand{\E}{{\mathbb E}}

\newcommand{\fonc}[5]{\begin{array}{ccccc}
#1 & : & #2 & \to & #3 \\
 & & #4 & \mapsto & #5 \\
\end{array}}

\let\leq=\leqslant
\let\geq=\geqslant

\newcommand{\tref}[1]{Theorem~\textup{\ref{#1}}}
\newcommand{\pref}[1]{Proposition~\textup{\ref{#1}}}
\newcommand{\cref}[1]{Corollary~\textup{\ref{#1}}}
\newcommand{\lref}[1]{Lemma~\textup{\ref{#1}}}

\newcommand{\dref}[1]{Definition~\textup{\ref{#1}}}

\DeclareMathOperator{\tr}{tr} 
 
\DeclareMathOperator{\Vol}{Vol}
\everymath{\displaystyle}

\title{On the density of cyclotomic lattices constructed from codes}

\author{Philippe Moustrou}
\date{\today This study has been carried out with financial support from the French State, managed by the French National Research Agency (ANR) in the frame of the "Investments for the future" Programme IdEx Bordeaux - CPU (ANR-10-IDEX-03-02)} 
\address{Institut de Math\'ematiques de Bordeaux, UMR 5251, universit\'e de Bordeaux, 351 cours de la Lib\'eration, 33400 Talence, France.}
\email{philippe.moustrou@u-bordeaux.fr}

\begin{document}

\begin{abstract}
Recently, Venkatesh improved the best known lower bound for lattice sphere packings by a factor $\log\log n$ for infinitely many dimensions $n$. Here we prove an effective version of this result, in the sense that we exhibit, for the same set of dimensions, finite families of lattices containing a lattice reaching this bound. Our construction uses codes over cyclotomic fields, lifted to lattices via Construction A. 
\end{abstract}

\maketitle

\section{Introduction}

The sphere packing problem in Euclidean spaces asks for the biggest proportion of space that can be filled by a collection of balls with disjoint interiors having the same radius. Here we focus on \textit{lattice} sphere packings, where the centers of the balls are located at the points of a lattice, and we denote by $\Delta_n$ the supremum of the density that can be achieved by such a packing in dimension $n$. Let us recall that the exact value of $\Delta_n$ is known only for dimensions up to $8$  \cite{Conway:1987:SLG:39091} and for dimension $24$ (\cite{MR2600869}). For other dimensions, only lower and upper bounds are known. Moreover, asymptotically, the ratio between these bounds is exponential.

Here we focus on lower bounds. The first important result goes back to the celebrated Minkowski-Hlawka theorem \cite{MR0009782}, stating the inequality $\Delta_n\geq  \frac{\zeta(n)}{2^{n-1}}$ for all $n$, where $\zeta(n)$ denotes the Riemann zeta function. Later, Rogers \cite{MR0022863} improved this bound by a linear factor: he showed that $\Delta_n\geq \frac{cn}{2^{n}}$ for every $n\geq 1$, with $c\approx 0.73$. The constant $c$ was successively improved by Davenport and Rogers \cite{MR0021018}  ($c=1.68$), Ball \cite{MR1191572} ($c=2$) and Vance \cite{MR2803798} ($c=2.2$ when $n$ is divisible by $4$). Recently  Venkatesh has obtained a more dramatic improvement \cite{MR3044452}, showing that for $n$ big enough, $\Delta_n\geq \frac{65963n}{2^{n}}$. Most importantly, he proves that for infinitely many dimensions $n$, $\Delta_n\geq \frac{n\log\log n}{2^{n+1}}$, thus improving for the first time upon the linear growth of the numerator. 

Unfortunately, all these results are of existential nature: their proofs are non constructive by essence, due to the fact that they generally use random arguments over infinite families of lattices. It is then natural to ask for effective versions of these results. It is worth to explain what we mean here by effectiveness. Indeed, designing a practical algorithm, i.e running in polynomial time in the dimension, to construct dense lattices appears to be out of reach to date. More modestly, one aims at exhibiting finite and explicit sets of lattices, possibly of exponential size, in which one is guaranteed to find a dense lattice. 

In this direction, the first to give an effective proof of Minkowski-Hlawka theorem was Rush \cite{MR1022304}. Later, Gaborit and Zémor \cite{MR2339326} provided an effective analogue of Roger's bound for the dimensions of the form $n=2p$ with $p$ a big enough prime number. In both constructions, the lattices are lifted from codes over a finite field, and run in sets of size of the form $\exp(k n \log n)$, with $k$ a constant.

Let us now explain with more details two ingredients that play a crucial role in the proofs of the results above.
The first one is \textit{Siegel's mean value theorem} \cite{MR0012093} which in particular states that, on average over the set $\mathcal{L}$ of $n$-dimensional lattices of volume $1$,
$$\E_\mathcal{L}[|B(r)\cap (\Lambda\setminus\{0\})|]=\Vol(B(r)).$$
It follows that, if $\Vol(B(r))<1$, then there exists a lattice $\Lambda\in \mathcal{L}$ such that $B(r)\cap (\Lambda\setminus\{0\})=\emptyset$, i.e such that the minimum norm $\mu$ of its non zero vectors is greater than $r$. The density of the sphere packing associated to $\Lambda$ then satisfies 
$$\Delta(\Lambda)= \frac{\Vol(B(\mu))}{2^n}>\frac{1}{2^n}.$$
It is worth to point out that the same reasoning holds if $\Vol(B(r))<2$, because lattice vectors of given norm come by pairs $\{\pm x\}$. From this simple remark we get
$$\Delta_n>\frac{2}{2^{n}},$$
which is essentially Minkowski-Hlawka bound.

The second idea follows almost immediately from the previous observation: considering lattices affording a group of symmetries larger than the trivial $\{\pm \text{Id}\}$ should allow to replace the factor $2$ in the numerator by a greater value. To this end, one needs a family of lattices, invariant under the action of a group, for which an analogue of Siegel's mean value theorem holds. This idea is exploited in \cite{MR2339326}, \cite{MR2803798} and \cite{MR3044452}. In particular, this is how Venkatesh obtains the extra $\log \log n$ term, by considering cyclotomic lattices, i.e lattices with an additional structure of $\Z[\zeta_m]$-modules. It turns out that, for a suitable choice of $m$, one can find such lattices in dimension $n=O(\frac{m}{\log\log m})$. 

In this paper, we consider cyclotomic lattices constructed from codes, in order to deal with finite families of lattices. To be more precise, the codes we take are the preimages through the standard surjection associated to a prime ideal $\mathfrak{P}$ of $\Q[\zeta_m]$
$$\Z[\zeta_m]^2 \to (\Z[\zeta_m]/\mathfrak{P})^2 $$
 of all one dimensional subspaces over the residue field $\Z[\zeta_m]/\mathfrak{P}$.

 Our approach is simpler and more straightforward than the previous ones in several respects. On one hand, the analogue of Siegel's mean value theorem in our situation boils down to a simple counting argument on finite sets (see \lref{Moyenne}). On the other hand, the group action, which is, as in \cite{MR2339326}, that of a cyclic group, is in our case  easier to deal with, because it is a free action. As a consequence, we can cope with arbitrary orders $m$, while Gaborit and Z\'emor only consider prime orders.

Our main theorem is  an effective version of Venkatesh's result:
\begin{theo}\label{theointro}
For infinitely many dimensions $n$, a lattice $\Lambda$ such that its density $\Delta(\Lambda)$ satisfies
$$\Delta(\Lambda)\geq \frac{0.89n \log \log n}{2^n}$$
can be constructed with $\exp(1.5n\log n(1+o(1))$ binary operations. 
\end{theo}

This result follows from a more general analysis of the density 
on average of the elements in the families of $m$-cyclotomic lattices described above, see \tref{Dens Général} and \pref{Complex} for precise statements.

A lattice $\Lambda$ is said to be \textit{symplectic} if there exists an isometry $\sigma$ exchanging $\Lambda$ and its dual lattice, and such that $\sigma^2=-\text{Id}$. Symplectic lattices are closely related to principally polarized Abelian varieties. In \cite{autissier:hal-01022915}, Autissier has adapted Venkatesh's approach to prove the existence of symplectic lattices with the same density. We show that, with some slight modifications, our construction leads to symplectic lattices, thus providing an effective version of Autissier's result (see \tref{Dens Général Symp} and \cref{Cor Ven Symp} ).

The article is organized as follows: Section 2 recalls basics notions about lattices and cyclotomic fields, and introduces the construction of cyclotomic lattices from codes. In Section 3 we state and prove the main results discussed above. Section 4 is dedicated to the case of symplectic lattices. 

\subsection*{Acknowledgements}  I am most grateful to Christine Bachoc for introducing me to this problem, and for her support all along this work. I would also like to thank Arnaud P\^echer and Gilles Z\'emor for fruitful discussions, and Pascal Autissier for useful remarks that lead to improvements on the first version of the paper.

\section{Notations and preliminaries}

\subsection{Lattices in Euclidean spaces}

Let $E$ be a Euclidean space equipped with the scalar product $\langle  , \rangle$. We denote by $||.||$ the norm associated to this scalar product, by $n$ the dimension of $E$, and by $B(r)$ the closed ball of radius $r$ in E: $$B(r)=\{x\in E, ||x||\leq r\}.$$
By Stirling formula, we have
$$\Vol(B(1))=\frac{\pi^{\frac{n}{2}}}{\Gamma(\frac{n}{2}+1)}\sim \frac{1}{\sqrt{n\pi}}\left(\sqrt{\frac{2\pi e}{n}}\right)^n$$
where $f\sim g$ means $\lim_{n\to\infty} f/g=1$. Thus, if $\Vol(B(r))=V$, we get that 
\begin{equation}
 \label{EquivRayon} r\sim \sqrt{\frac{n}{2\pi e}} V^\frac{1}{n}.
\end{equation}

A \textit{lattice} $\Lambda\in E$ is a free discrete $\Z$-module of rank $n$ (for a general reference on lattices, see e.g \cite{Conway:1987:SLG:39091}). 
A \textit{fundamental region} of $\Lambda$ is a region $\mathcal{R}\subset E$ such that for any $\lambda\neq\lambda'\in\Lambda$, the measure of $(\lambda + \mathcal{R})\cap (\lambda' + \mathcal{R})$ is $0$, and $E=\bigcup_{\lambda\in\Lambda} (\lambda + \mathcal{R})$.
The \textit{volume} $\Vol(\Lambda)$ of $\Lambda$ is defined as the volume of any of its fundamental region.
The \textit{Voronoi region} of $\Lambda$ is the particular fundamental region:
$$\mathcal{V}=\mathcal{V}_\Lambda=\{z\in E, \forall \text{ } x\in  \Lambda, ||z-x||\geq ||z||\}.$$ 
We denote by $\mu$ the \textit{minimum} of $\Lambda$:
$$\mu=\mu_\Lambda=\min\{{||x||, x\in \Lambda \setminus \{0\}}\}$$
and by $\tau$ its \textit{covering radius}:
$$\tau=\tau_\Lambda=\sup_{z\in E}\inf_{x\in \Lambda}||z-x||.$$
Taking balls of radius $\mu/2$ centered at the points of $\Lambda$, we get a \textit{packing} in $E$, i.e a set of spheres with pairwise disjoint interiors. The \textit{density} of this packing is given by 
$$\Delta(\Lambda)=\frac{\Vol(B(\mu))}{2^{n}\Vol(\Lambda)}.$$
Finally, we define $\Lambda^\#$, the \textit{dual lattice} of the lattice $\Lambda$:
$$\Lambda^\# =\{ x \in E , \forall \hspace{1mm} y\in \Lambda , \langle x , y \rangle \in \Z\}.$$

\subsection{Cyclotomic fields}

Let $K$ be the cyclotomic field $\Q[\zeta_m]$, where $\zeta_m$ is a primitive $m$-th root of unity. This is a totally imaginary field of degree $\phi(m)$ over $\Q$. Let us define $K_{\R}=K \otimes _\Q \R$. The trace form $\tr(x\overline{y})$ where $\tr$ denotes the trace form of the number field $K$ induces a scalar product on $K_{\R}$, denoted by $\langle , \rangle$, giving $K_{\R}$ the structure of a Euclidean space of dimension $\phi(m)$. We refer to \cite{MR1421575} for general properties of cyclotomic fields.

For every fractional ideal $\mathfrak{A}$, we will use the same notation $\mathfrak{A}$ for the lattice in $K_\R$  which is the image of $\mathfrak{A}$ under the natural embedding $K\to K_{\R}$. We will need informations about lattices defined by fractional ideals of $K$.

The volume of  $\mathcal{O}_K$ is by definition the square root of the absolute value of the discriminant $d_K$ of $K$. 
It is well known  (e.g \cite{MR1421575}) that for the cyclotomic fields
\begin{equation}\label{Disc} |d_K|=\frac{m^{\phi(m)}}{\displaystyle{\prod_{\substack{ l\in \mathbb{P} \\ l|m}} l^{\phi(m)/(l-1)}}} \end{equation}
where $\P$ is the set of prime numbers.

It is easy to see that the minimum of $\mathcal{O}_K$ is $\sqrt{\phi(m)}$: indeed $||1||=\sqrt{\phi(m)}$ and the arithmetic geometric inequality gives $||x|| \geq \sqrt{\phi(m)}$ for all  $x\in \mathcal{O}_K$. For the minimum and the covering radius of general fractional ideals, we will apply the following estimates:
\begin{lemm}[\cite{BayerFluckiger2006305}, propositions 4.1 and 4.2.]\label{BornesMinima}
Let $\mathfrak{A}$ be a fractional ideal of $K$, where $K$ is a number field of degree $n$ over $\Q$. Then we have :
\begin{enumerate}[(i)]
\item[\emph{(i)}] 
$\frac{\mu_{\mathfrak{A}}}{\Vol(\mathfrak{A})^{\frac{1}{n}}} \geq \frac{\sqrt{n}}{\sqrt{|d_K|}^{\frac{1}{n}}},$
\item[\emph{(ii)}]
$\frac{\tau_{\mathfrak{A}}}{\Vol(\mathfrak{A})^{\frac{1}{n}}} \leq \frac{\sqrt{n}}{2}{\sqrt{|d_K|}^{\frac{1}{n}}}.$
\end{enumerate}
\end{lemm}

\subsection{Cyclotomic lattices constructed from codes}

A standard construction of lattices lifts codes over $\F_p$ to sublattices of $\Z^n$, this is the well known \textit{Construction A} (see \cite[Chapter 7]{Conway:1987:SLG:39091}). Here we will deal with a slightly more general construction in the context of cyclotomic fields.

Let us consider as before $K=\Q[\zeta_m]$ and $K_\R$ the Euclidean space associated with $K$. Let $\mathfrak{P}$ be a prime ideal of $\mathcal{O}_K$ lying over a prime number $p$ which does not divide $m$. Then the quotient $F=\mathcal{O}_K/\mathfrak{P}$ is a finite field of cardinality $q=p^f$.

Let $E=K_\R^s$. We still denote by $\langle ,\rangle$ the scalar product  $\langle x ,y\rangle=\sum_{i=1}^s \langle x_i,y_i \rangle$ induced on the $s\phi(m)$-dimensional $\R$-vector space $E$ by that of $K_\R$. Let $\Lambda_0$ be a lattice in $E$ which is a  $\mathcal{O}_K$-submodule of $E$. We consider the canonical surjection
$$\pi:\Lambda_0 \to \Lambda_0 \big/ \mathfrak{P}\Lambda_0.$$

The norm $||.||$ on $E$ associated with $\langle , \rangle$ induces a weight on the quotient space $\Lambda_0 \big/ \mathfrak{P}\Lambda_0$: if $c\in \Lambda_0 \big/ \mathfrak{P}\Lambda_0$,
$$wt(c)=\min\{||z||,\pi(z)=c\}.$$

The quotient $\Lambda_0 \big/ \mathfrak{P}\Lambda_0$ is a vector space of dimension $s$ over the finite field $F$. We will call a $F$-subspace $C$ of $\Lambda_0 \big/ \mathfrak{P}\Lambda_0$ a \textit{code}. We denote by $k$ its dimension and by $d$ its minimal weight, with respect to the weight defined above. 
Finally we denote by $\Lambda_C$ the lattice obtained from $C$ 
$$\Lambda_C=\pi^{-1}(C)$$
and give in the following lemma a summary of its properties:
\begin{lemm}\label{Dens} Let $C$ be a code of $\Lambda_0 \big/ \mathfrak{P}\Lambda_0$ of dimension $k$ and minimal weight $d$. Then :
\begin{enumerate}[(i)]
\item[\emph{(i)}] The volume of $\Lambda_C$ is 
$$\Vol(\Lambda_C)={q^{s-k}}\Vol(\Lambda_0).$$
\item[\emph{(ii)}] The minimum of $\Lambda_C$ is $\mu_{\Lambda_C}=\min\{d,\mu_{\mathfrak{P}\Lambda_0}\}$. 
\item[\emph{(iii)}] If $d\leq \mu_{\mathfrak{P}\Lambda_0}$, the packing density of $\Lambda_C$ is:
$$\Delta(\Lambda_C)=\frac{\Vol(B(d))}{2^{n}{q^{s-k}}\Vol(\Lambda_0)},$$
where $n=s\phi(m)$ is the dimension of $E$.
\end{enumerate}
\end{lemm}

\begin{proof}
\begin{enumerate}[(i)]
\item The lattice $\pi^{-1}(C)$ contains the lattice $\mathfrak{P}\Lambda_0$ and we have:
$$|\pi^{-1}(C) / \mathfrak{P}\Lambda_0| = |C| = q^k,$$
so
$$\Vol(\Lambda_C)=\frac{1}{q^k}\Vol(\mathfrak{P}\Lambda_0)={q^{s-k}}\Vol(\Lambda_0).$$
\item and (iii) follow directly from the definitions. 
\end{enumerate}
\end{proof}

To conclude this subsection, we state a lemma that relates the Euclidean ball and the discrete ball $\overline{B}(r):=\{ c\in \Lambda_0 \big/ \mathfrak{P}\Lambda_0, wt(c)\leq r\}$. 

\begin{lemm}\label{LBoules}
Assuming $r< \frac{\mu_{\mathfrak{P}\Lambda_0}}{2}$, we have:
\begin{enumerate}[(i)]
\item[\emph{(i)}]\label{inter} $|\overline{B}(r)|=|\Lambda_0\cap B(r)|$

\item[\emph{(ii)}]\label{VolPoints} $\Vol(B(r-\tau_{\Lambda_0}))\leq |\overline{B}(r)|\Vol(\Lambda_0)\leq \Vol(B(r+\tau_{\Lambda_0}))$.
\item[\emph{(iii)}] If $\overline{B}(r)\cap (C\setminus\{0\})=\emptyset$, then 
\begin{equation}\label{DensRay}
\Delta(\Lambda_C)>\frac{\Vol(B(r))}{2^{n}{q^{s-k}}\Vol(\Lambda_0)}.
\end{equation}

\end{enumerate}
\end{lemm} 
\begin{proof}
\begin{enumerate}[(i)]
\item Let $c\in \Lambda_0 \big/ \mathfrak{P}\Lambda_0$ such that $wt(c)\leq r$. We want to prove that $c$ has exactly one representative $x\in\Lambda_0$  which satisfies $||x||\leq r$. Indeed, if $y\in \Lambda_0$ with $y\neq x$ and $\pi(y)=\pi(x)=c$, we have $y=x+z$ with $z\in \mathfrak{P}\Lambda_0\setminus\{0\}$. Then $||x-y||=||z||\geq \mu_{\mathfrak{P}\Lambda_0}> 2r$, a contradiction.

\item Let us consider
$$A=\bigcup_{x\in\Lambda_0\cap B(r)} (x+\mathcal{V}_{\Lambda_0})$$
where $\mathcal{V}_{\Lambda_0}$ is the Voronoi region of $\Lambda_0$.
The volume of $A$ is 
$$\Vol(A)=|\Lambda_0\cap B(r)|\Vol(\Lambda_0)=|\overline{B}(r)|\Vol(\Lambda_0)$$
so the wanted inequalities will follow from the inclusions
$$B(r-\tau_{\Lambda_0})\subset A\subset B(r+\tau_{\Lambda_0}).$$ 

Let us start with the second inclusion.
If $z\in x+\mathcal{V}_{\Lambda_0}$, by definition of the covering radius, we have
$$||z-x||\leq \tau_{\Lambda_0},$$
so if $||x||\leq r$, $||z||\leq r+\tau_{\Lambda_0}$.
For the first inclusion, let $y$ be such that $||y||\leq r-\tau_{\Lambda_0}$. If $x$ denotes the closest point to $y$ in $\Lambda_0$, we have $y\in x + \mathcal{V}_{\Lambda_0}$ and $||x||\leq ||y|| + ||x-y|| \leq r$, so that $y\in A$.
\item It follows directly from \lref{Dens}.

\end{enumerate} 
\end{proof}

\section{The density of cyclotomic lattices constructed from codes}

In this section, we introduce a certain family of lattices obtained from codes as described in the previous subsection, and show that for high dimensions, this family contains lattices having good density.

As before, $K=\Q[\zeta_m]$, $F=\mathcal{O}_K/\mathfrak{P}\simeq \F_q$. Let us set $s=2$ and consider the Euclidean space $E=K_\R^2$, of dimension $2\phi(m)$, in which we fix $\Lambda_0=\mathcal{O}_K^2$.  
\begin{defi}\label{Famille}
We denote by $\mathcal{C}$ the set of the $(q+1)$ $F$-lines of $\Lambda_0 \big/ \mathfrak{P}\Lambda_0=F^2$, and by $\mathcal{L}_{\mathcal{C}}$ the set of lattices of $E$ constructed from the codes in $\mathcal{C}$:
$$\mathcal{L}_{\mathcal{C}}=\{\Lambda_C,C\in \mathcal{C}\}.$$  
\end{defi}

The following lemma evaluates the average of the value of $|\overline{B}(r)\cap C\setminus\{0\}|$ over the family $\mathcal{C}$:
\begin{lemm}\label{Moyenne}
We have:  
$$\E(|\overline{B}(r)\cap (C\setminus\{0\})|) < \frac{|\overline{B}(r)|}{q}.$$
\end{lemm}
\begin{proof}
It is a straightforward computation: 
$$\begin{aligned} \E(|\overline{B}(r)\cap (C\setminus\{0\})|)&= \frac{1}{|\mathcal{C}|}\sum_{C\in\mathcal{C}} |\overline{B}(r)\cap (C\setminus\{0\})| \\ &= \frac{1}{|\mathcal{C}|}\sum_{C\in\mathcal{C}} \sum_{\substack{ c\in C  \\ 0<wt(c)\leq r }}1
\\ &= \frac{1}{|\mathcal{C}|}\sum_{c\in \overline{B}(r)\setminus\{0\}} |\{C\in \mathcal{C} \text{ , } c\in C\}| . \end{aligned}$$
There is exactly one line passing through every non zero vector in $F^2$. So
$$\E(|\overline{B}(r)\cap (C\setminus\{0\})|)= \frac{|\overline{B}(r)\setminus\{0\}|}{|\mathcal{C}|}< \frac{|\overline{B}(r)|}{q}.$$
\end{proof}

From now on, $q$ will vary with $m$, so we adopt the notation $q_m$ instead of $q$.
We show that the family $\mathcal{L}_\mathcal{C}$ of lattices contains, when $m$ is big enough and when $q_m$ grows in a suitable way with $m$, lattices having high density.

\begin{theo}\label{Dens Général} For every $1>\varepsilon >0$, if $\phi(m)^2 m=o({q_m}^{\frac{1}{\phi(m)}})$, then for $m$ big enough, the family of lattices $\mathcal{L}_\mathcal{C}$ contains a lattice $\Lambda\subset \R^{2\phi(m)}$ satisfying
$$\Delta(\Lambda)> \frac{(1-\varepsilon)m}{2^{2\phi(m)}}.$$

\end{theo}

We start with a technical lemma.

\begin{lemm}\label{Conds} Let $\rho_m=\sqrt{\frac{\phi(m)}{\pi e}}(q_m\Vol(\Lambda_0))^{\frac{1}{2\phi(m)}}$. If $\phi(m)^2 m=o({q_m}^{\frac{1}{\phi(m)}})$, then
\begin{enumerate}[(i)]
\item[\emph{(i)}] $\lim_{m\to \infty}\frac{\phi(m)\tau_{\Lambda_0}}{\rho_m}=0$,
\item[\emph{(ii)}] For $m$ big enough, $\rho_m< \frac{\mu_{\mathfrak{P}\Lambda_0}}{2}$.
\end{enumerate}

\end{lemm}

\begin{proof}
\begin{enumerate}[(i)]
\item We have:
$$\frac{\phi(m)\tau_{\Lambda_0}}{\rho_m}=\frac{\sqrt{\pi e \phi(m)} \tau_{\Lambda_0}}{(q_m\Vol(\Lambda_0))^{\frac{1}{2\phi(m)}}}.$$

Since $\Lambda_0=\mathcal{O}_K\times\mathcal{O}_K$, we have  $\tau_{\Lambda_0}=\sqrt{2}\tau_{\mathcal{O}_K}$ and $\Vol(\Lambda_0)=\Vol(\mathcal{O}_K)^2$. Then, by (ii) of \lref{BornesMinima}, 
$$\frac{\tau_{\mathcal{O}_K}}{\Vol(\mathcal{O}_K)^{\frac{1}{\phi(m)}}}\leq \frac{\sqrt{\phi(m)}}{2}{|d_K|}^{\frac{1}{2\phi(m)}}.$$
Applying $|d_K|\leq m^{\phi(m)}$ (following \eqref{Disc}), we obtain

$$\frac{\tau_{\mathcal{O}_K}}{\Vol(\mathcal{O}_K)^{\frac{1}{\phi(m)}}}\leq\frac{\sqrt{m\phi(m)}}{2}.$$

So
$$\frac{\phi(m)\tau_{\Lambda_0}}{\rho_m}\leq\sqrt{\frac{\pi e}{2}}{ \phi(m) \sqrt{m}}\,{q_m^{-\frac{1}{2\phi(m)}}}$$
which tends to $0$ when $m$ goes to infinity, by hypothesis.

\item We have:
$$\begin{aligned} \rho_m=\sqrt{\frac{\phi(m)}{\pi e}}\big(q_m\Vol(\Lambda_0)\big)^{\frac{1}{2\phi(m)}} &\leq \frac{1}{2} \sqrt{\phi(m)} q_m ^{\frac{1}{2\phi(m)}} {|d_K|}^{\frac{1}{2\phi(m)}} \\ &\leq \frac{1}{2} \sqrt{\phi(m)} q_m ^{\frac{1}{2\phi(m)}} \sqrt{m} \end{aligned}.$$

Because $\mathfrak{P}\Lambda_0=\mathfrak{P}\times \mathfrak{P}$, $\mu_{\mathfrak{P}\Lambda_0}=\mu_{\mathfrak{P}}$. Then, by (i) of \lref{BornesMinima}, since $\Vol(\mathfrak{P})=q_m \sqrt{|d_K|}$, 
$$\mu_{\mathfrak{P}}\geq q_m^{\frac{1}{\phi(m)}} \sqrt{\phi(m)}.$$
The hypothesis on $q_m$ ensures in particular that for $m$ big enough, we have $m<q_m^{\frac{1}{\phi(m)}}$, and thus
$$\rho_m<\frac{1}{2} \sqrt{\phi(m)} q_m ^{\frac{1}{\phi(m)}} \leq \frac{\mu_{\mathfrak{P}\Lambda_0}}{2} .$$

\end{enumerate}
\end{proof}

Now we can prove \tref{Dens Général}.

\begin{proof}[Proof of \tref{Dens Général}] 
Let us fix $1>\varepsilon >0$. Let $r_m>0$ be the radius such that $\Vol(B_{r_m})=(1-\varepsilon)m q_m\Vol(\Lambda_0)$. By \eqref{EquivRayon}, $r_m\sim \rho_m$, where $\rho_m$ is the radius defined in \lref{Conds}.
Applying \lref{Moyenne}, we get
$$ \E(|\overline{B}(r_m)\cap (C\setminus\{0\})|) < \frac{|\overline{B}(r_m)|}{q_m}.$$
Because $r_m\sim \rho_m$, by (ii) of \lref{Conds}, $r_m<\frac{\mu_{\mathfrak{P}\Lambda_0}}{2}$, so we can apply (ii) of \lref{LBoules}, so that
$$ \begin{aligned} \E(|\overline{B}(r_m)\cap (C\setminus\{0\})|) < \frac{\Vol(B(r_m+\tau_{\Lambda_0}))}{q_m\Vol(\Lambda_0)} &=\frac{\Vol(B(r_m))}{q_m\Vol(\Lambda_0)}\left(1+\frac{\tau_{\Lambda_0}}{r_m}\right)^{2\phi(m)} \\ &= (1-\varepsilon)m\left(1+\frac{\tau_{\Lambda_0}}{r_m}\right)^{2\phi(m)}. \end{aligned}$$
Now applying (i) of \lref{Conds}, we have $\lim_{m\to\infty}\left(1+\frac{\tau_{\Lambda_0}}{r_m}\right)^{2\phi(m)}=1$, and so, for $m$ big enough, 
\begin{equation}\label{IneqMoy}
\E(|\overline{B}(r_m)\cap (C\setminus\{0\})|) < m.
\end{equation} 

Now comes the crucial argument involving the action of the $m$-roots of unity. From \eqref{IneqMoy}, there is at least one code $C$ in $\mathcal{C}$ which satisfies $|\overline{B}(r_m)\cap (C\setminus\{0\})| < m$. Because the codes we consider are stable under the action of the $m$-roots of unity, which preserves the weight of the codewords, and because the length of every non zero orbit under this action is $m$, we can conclude that  $\overline{B}(r_m)\cap (C\setminus\{0\}) = \emptyset$, and so by (iii) of \lref{LBoules} that,
$$\Delta(\Lambda_C)>\frac{\Vol(B(r_m))}{2^{2\phi(m)}{q_m}\Vol(\Lambda_0)}=\frac{(1-\varepsilon)m}{2^{2\phi(m)}}.$$
\end{proof}

\tref{Dens Général} shows that for every big enough dimension of the form $n=2\phi(m)$ our construction provides lattices having density approaching $\frac{m}{2^n}$, thus larger than $\frac{cn}{2^n}$ with $c=1/2$. A particular sequence of dimensions leads to a better lower bound:

\begin{coro}\label{CorVen} For infinitely many dimensions, the family $\mathcal{L}_\mathcal{C}$ contains a lattice $\Lambda\subset \R^n$ satisfying
$$\Delta(\Lambda)\geq \frac{0.89 n \log\log n}{2^{n}}.$$

\end{coro}

\begin{proof}

To get the optimal gain between $m$ and $2\phi(m)$, we take $m=\prod_{ \substack{ l\in \mathbb{P} \\ l\leq X }} l $, where $X$ is a positive real number, which tends to infinity. Thanks to Mertens' theorem \cite{MR1574590}, we can evaluate: 
\begin{equation}\label{equiv}
\frac{m}{\phi(m)}\sim {e^{\gamma}}\log\log m.
\end{equation}
where $\gamma$ is the Euler-Mascheroni constant which satisfies $\gamma>0.577$. 

So we get 
\begin{equation}\label{Mertens}
m\sim {\phi(m)}{e^{\gamma}}\log\log m\sim \frac{e^{\gamma}}{2}n\log\log n.
\end{equation} 

Let us set $\delta:= 2e^{-\gamma}0.89$. Because $\frac{e^{\gamma}}{2}>0.89$, $\delta<1$. Then by \tref{Dens Général}, we get a lattice $\Lambda\subset \R^n$ such that 
 $$\Delta(\Lambda)> \frac{\delta m}{2^{n}}.$$
So by \eqref{Mertens}, for $m$ big enough, 
 $$\Delta(\Lambda)\geq \frac{0.89 n \log \log n}{2^{n}}.$$
 
 \end{proof}
 
Finally we evaluate the complexity of constructing a lattice $\Lambda$ with the desired density: 

\begin{prop}\label{Complex}
Let $n=2\phi(m)$. For every $1>\varepsilon>0$, the construction of a lattice $\Lambda\subset\R^n$ satisfying $$\Delta(\Lambda)> \frac{(1-\varepsilon)m}{2^{2\phi(m)}}$$ requires $\exp(1.5n\log n(1+o(1))$ binary operations. 
\end{prop}

We need to find a prime ideal $\mathfrak{P}$ such that  $q_m=|\mathcal{O}_K/\mathfrak{P}|$ satisfies the condition required in $\tref{Dens Général}$. Let us recall that $q_m=p_m^{f_m}$ where $p_m$ is the prime number lying under $\mathfrak{P}$, and $f_m$ is the order of $p_m$ in the group $(\Z/m\Z)^*$ (see \cite{MR1421575}). We will restrict our attention to the case $f_m=1$, i.e when $p_m=1\mod m$. In that case, $p_m$ decomposes totally  in $\Q[\zeta_m]$, and $q_m=p_m$.
We use Siegel-Walfisz theorem in order to give an upper bound for the smallest such prime number:
\begin{lemm}\label{Para} For $m$ big enough, there is a prime number $p_m$ congruent to $1\mod m$ such that: 
$$ \frac{1}{2}(m^3\log m)^{\phi(m)} \leq p_m \leq (m^3\log m)^{\phi(m)}. $$

\end{lemm}

\begin{proof}
Let us denote by $\pi(x,m,a)$ the number of primes $p<x$ such that $p=a \mod m$. Siegel-Walfisz theorem (see \cite{book:1211416}) gives that for any $A>0$:
$$\pi(x,m,a)=\frac{Li(x)}{\phi(m)}+\mathcal{O}(\frac{x}{(\log x)^A}),$$
where the implied constant depends only on $A$, and $Li(x)=\int_2^x \frac{dt}{\log t}$.
Applying this theorem to $x=(m^3\log m)^{\phi(m)}$, $a=1$, and $A=2$ we get
$$\pi(x,m,1)-\pi(x/2,m,1)= \frac{1}{\phi(m)}\int_{x/2}^x \frac{dt}{\log t} +\mathcal{O}(\frac{x}{(\log x)^2}).$$
We have $\frac{1}{\phi(m)}\int_{x/2}^x \frac{dt}{\log t}> \frac{x}{2\phi(m)\log x}$, which grows faster than the error term since $\log x \sim 3\phi(m)\log(m)$, and thus ensures the existence of a prime $p_m$ between $x/2$ and $x$.
\end{proof}

\begin{proof}[Proof of \pref{Complex}]
Applying \lref{Para}, the complexity of finding $q_m$ satisfying the condition of \tref{Dens Général} is $$\mathcal{O}(m^3\log m)^{\phi(m)}=e^{3\phi(m)\log(m)(1+o(1))}=e^{1.5n\log(n)(1+o(1))}.$$
The corresponding family of lattices $\mathcal{L}_\mathcal{C}$ has $q_m + 1$ elements. By construction, each of these lattices is generated by vectors with coefficients which are polynomial in $n$. So, the cost of computing their density, which can be done with $2^{O(n)}$ operations, following \cite{Hanrot_algorithmsfor}, is negligible compared with the enumeration of the family. 
\end{proof}

\section{Symplectic cyclotomic lattices}

For a survey about symplectic lattices, we refer to \cite{Berge_symplecticlattices}. Here we briefly introduce this notion.
 
Let $E$ be a Euclidean space, and $\Lambda$ a lattice in $E$. Then an \textit{isoduality} is an isometry $\sigma$ of $E$ such that $\sigma(\Lambda)=\Lambda^\#$. If $\Lambda$ affords an isoduality, then it is called isodual. If moreover $\sigma$ satisfies $\sigma^2=-\text{Id}$, then $\Lambda$ is called \textit{symplectic}. 

Now we explain how to change the lattice $\Lambda_0$ in such a way that our construction provides symplectic lattices.

Let
$$\Lambda_0=\alpha^{-1} \mathcal{O}_K \times \alpha\,\mathfrak{P}^{-1}\mathcal{O}_K^\#,$$
where $\alpha=(q|d_K|)^\frac{1}{2\phi(m)}$.
The volume of $\Lambda_0$ is now
\begin{equation}\label{VolBaseSymp}
\Vol(\Lambda_0)=\Vol(\mathcal{O}_K)\Vol(\mathfrak{P}^{-1}\mathcal{O}_K^\#)=\frac{\Vol(\mathcal{O}_K)\Vol(\mathcal{O}_K^\#)}{q}=\frac{1}{q}.
\end{equation}

Let us define the map 
$$\fonc{\sigma}{K_\R^2}{K_\R^2}{(x_1,x_2)}{(-\overline{x_2},\overline{x_1})}.$$
It is clear that $\sigma$ is an isometry, and that $\sigma^2=-\text{Id}$. 

In the following lemma, we show that the lattices we defined in \dref{Famille} are now symplectic: 
\begin{lemm}\label{PropSymp}

If $C$ is a $F$-line of $\Lambda_0/\mathfrak{P}\Lambda_0$, then the lattice $\Lambda_C$ is symplectic.
\end{lemm}

\begin{proof}
Let us prove that $\sigma(\Lambda_C)\subset\Lambda_C^\#$. Let us take $(x_1,x_2)\in \Lambda_C$. We have to show that for every $(y_1,y_2)\in\Lambda_C$, $\langle \sigma(x_1,x_2),(y_1,y_2)\rangle\in\Z$, that is
\begin{equation}\label{ToshowSymp}
\tr(-x_2y_1)+\tr(x_1y_2)\in\Z.
\end{equation}
According to the definition of $C$, we have $C=F(u_1,u_2)$ with $u_1\in\alpha^{-1}\mathcal{O}_K$ and $u_2\in \alpha\,\mathfrak{P}^{-1}\mathcal{O}_K^\#$. So there exists $\lambda,\mu\in \mathcal{O}_K$ such that 
$$
\begin{cases}
x_1=\lambda u_1 \mod \alpha^{-1}\mathfrak{P}
\\
 x_2=\lambda u_2 \mod\alpha \, \mathcal{O}_K^\#
\end{cases}
\text{ and }
\begin{cases}
y_1=\mu u_1 \mod \alpha^{-1}\mathfrak{P}
\\
y_2=\mu u_2 \mod \alpha \, \mathcal{O}_K^\#
\end{cases}.
$$
This implies that 
$$\tr(x_1y_2)=\tr(\lambda\mu u_1 u_2) \mod \Z$$
and
$$\tr(x_2y_1)=\tr(\lambda\mu u_1 u_2) \mod \Z,$$

so that \eqref{ToshowSymp} is satisfied.

To conclude the proof it is enough to notice that $\Vol(\Lambda_C)=q\Vol(\Lambda_0)=1$, which implies $\sigma(\Lambda_C)=\Lambda_C^\#$.
\end{proof}

We again consider the set $\mathcal{C}$ of lines of $\Lambda_0/\mathfrak{P}\Lambda_0$. It is clear that the result of \lref{Moyenne} remains valid for this new family of codes. The general strategy underlying the proof of \tref{Dens Général} applies to the family of lattices associated to these codes, so that we get analogues in this context :

\begin{theo}\label{Dens Général Symp} For every $1>\varepsilon >0$, if $\phi(m)^2 m=o({q_m}^{\frac{1}{\phi(m)}})$, then for $m$ big enough, the family of \textbf{{symplectic}} lattices $\mathcal{L}_\mathcal{C}$ contains a lattice $\Lambda\subset \R^{2\phi(m)}$ satisfying
$$\Delta(\Lambda)> \frac{(1-\varepsilon)m}{2^{2\phi(m)}}.$$

\end{theo}

\begin{coro}\label{Cor Ven Symp} For infinitely many dimensions, the family $\mathcal{L}_\mathcal{C}$ contains a \textbf{symplectic} lattice $\Lambda\subset \R^n$ satisfying
$$\Delta(\Lambda)\geq \frac{0.89 n \log\log n}{2^{n}}.$$

\end{coro}

The proofs of \tref{Dens Général Symp} and \cref{Cor Ven Symp} are similar to those of \tref{Dens Général} and \cref{CorVen}. However, we need to prove that \lref{Conds} still holds, even if we changed $\Lambda_0$:

\begin{lemm}\label{Conds2} Let $\rho_m=\sqrt{\frac{\phi(m)}{\pi e}}(q_m\Vol(\Lambda_0))^{\frac{1}{2\phi(m)}}=\sqrt{\frac{\phi(m)}{\pi e}}$. 

If $\phi(m)^2 m=o({q_m}^{\frac{1}{\phi(m)}})$, then
\begin{enumerate}[(i)]
\item[\emph{(i)}] $\lim_{m\to \infty}\frac{\phi(m)\tau_{\Lambda_0}}{\rho_m}=0$,
\item[\emph{(ii)}] For $m$ big enough, $\rho_m< \frac{\mu_{\mathfrak{P}\Lambda_0}}{2}$.
\end{enumerate}

\end{lemm}

\begin{proof}
\begin{enumerate}[(i)]
\item We have:
$$\frac{\phi(m)\tau_{\Lambda_0}}{\rho_m}={\sqrt{\pi e \phi(m)} \tau_{\Lambda_0}}.$$

Let us set $\mathfrak{A}_1=\alpha^{-1} \mathcal{O}_K$ and $\mathfrak{A}_2= \alpha\,\mathfrak{P}^{-1}\mathcal{O}_K^\#$. Then $\Lambda_0=\mathfrak{A}_1\times\mathfrak{A}_2$ and the covering radius of $\Lambda_0$ is $\tau_{\Lambda_0}=\sqrt{\tau_{\mathfrak{A}_1}^2+\tau_{\mathfrak{A}_2}^2}$. So we have to bound both covering radii $\tau_{\mathfrak{A}_1}$ and $\tau_{\mathfrak{A}_2}$. Applying (ii) of \lref{BornesMinima}, and because $\Vol(\mathfrak{A}_1)=\Vol(\mathfrak{A}_2)=\frac{1}{\sqrt{q}} $, we have, for $i\in\{1,2\}$,
$$\tau_{\mathfrak{A}_i}\leq \frac{\sqrt{\phi(m)}}{2}|d_K|^{\frac{1}{2\phi(m)}}{q^{-\frac{1}{2\phi(m)}}}\leq \frac{\sqrt{m\phi(m)}\,q^{-\frac{1}{2\phi(m)}}}{2}  
.$$
So

$$\tau_{\Lambda_0}\leq \sqrt{2}\max\{\tau_{\mathfrak{A}_1},\tau_{\mathfrak{A}_2}\}\leq {\sqrt{m\phi(m)}}{q^{-\frac{1}{2\phi(m)}}}$$

and finally

$$\frac{\phi(m)\tau_{\Lambda_0}}{\rho_m}\leq{\sqrt{\pi e}}\,{ \phi(m) \sqrt{m}}\,{q_m^{-\frac{1}{2\phi(m)}}}$$
which tends to $0$ when $m$ goes to infinity, by hypothesis.

\item 

Let us set $\mathfrak{B}_1=\alpha^{-1} \mathfrak{P}$ and $\mathfrak{B}_2= \alpha\,\mathcal{O}_K^\#$. Then $\mathfrak{P}\Lambda_0=\mathfrak{B}_1\times \mathfrak{B}_2$, and clearly $\mu_{\mathfrak{P}\Lambda_0}=\min\{\mu_{\mathfrak{B}_1},\mu_{\mathfrak{B}_2}\}$. Then, applying (i) of \lref{BornesMinima}, since $\Vol(\mathfrak{B}_1)=\Vol(\mathfrak{B}_2)=\sqrt{q} $, we have, for $i\in\{1,2\}$,
$$
\mu_{\mathfrak{B}_i}\geq \frac{\sqrt{\phi(m)} \,q^{\frac{1}{2\phi(m)}}}{|d_K|^{\frac{1}{2\phi(m)}}}\geq \frac{\sqrt{\phi(m)}\, q^{\frac{1}{2\phi(m)}}}{\sqrt{m}}   . $$
So 
$$\mu_{\mathfrak{P}\Lambda_0}\geq \frac{\sqrt{\phi(m)} \,q^{\frac{1}{2\phi(m)}}}{\sqrt{m}}.$$
The hypothesis on $q_m$ ensures in particular that for $m$ big enough, $m$ satisfies $\sqrt{m}<q_m^{\frac{1}{2\phi(m)}}$, and thus
$$\rho_m=\sqrt{\frac{\phi(m)}{\pi e}}<\frac{1}{2} \frac{\sqrt{\phi(m)} \,q^{\frac{1}{2\phi(m)}}}{m} \leq \frac{\mu_{\mathfrak{P}\Lambda_0}}{2} .$$

\end{enumerate}
\end{proof}

As the condition on the growth of $q_m$ does not change, the estimation for the complexity of construction in this context is the same:
\begin{prop}\label{Complex Symp}
Let $n=2\phi(m)$. For every $1>\varepsilon>0$, the construction of a \textbf{symplectic} lattice $\Lambda\subset\R^n$ satisfying $$\Delta(\Lambda)> \frac{(1-\varepsilon)m}{2^{2\phi(m)}}$$ requires $\exp(1.5n\log n(1+o(1))$ binary operations. 
\end{prop}

\bibliographystyle{alpha}
\bibliography{Bibliographie}

\end{document}